\documentclass[reqno,11pt, letterpaper]{amsart}
\PassOptionsToPackage{fleqn}{amsmath}
\RequirePackage{amsthm}
\RequirePackage{amsmath}
\RequirePackage{latexsym}
\RequirePackage{amssymb}
\RequirePackage{amscd}
\RequirePackage{epsfig}
\RequirePackage{graphics}
\RequirePackage{ifthen}
\RequirePackage{varioref}
\usepackage{changes}
\usepackage{soul}
\usepackage{xcolor}
\usepackage{color}
\usepackage[misc]{ifsym}
\usepackage{bm}
\usepackage{lipsum}
 \usepackage{verbatim}

\usepackage{epstopdf}

\definecolor{darkgreen}{rgb}{0.0625,0.64,0.0625}
\ifpdf
  \usepackage[
    pdftex,
    colorlinks,%
    linkcolor=blue,citecolor=blue,urlcolor=blue,
    hyperindex,%
    plainpages=false,%
    bookmarksopen,%
    bookmarksnumbered%
  ]{hyperref}
\else
  \usepackage{hyperref}
\fi
\usepackage{bookmark}
\allowdisplaybreaks[1]

\def\R{{\mathbb R}}

\theoremstyle{plain}
\newtheorem{thm}{Theorem}[section]
\newtheorem{lem}[thm]{Lemma}
\newtheorem{prop}[thm]{Proposition}
\newtheorem{cor}[thm]{Corollary}
\theoremstyle{definition}

\newtheorem{exmp}{Example}[section]

\def\ro#1{{\rm #1}}

\def\Bbb#1{{\mathbb#1}}

\def\R{\Bbb R}

\def\Rx{\R\mkern1mu^}

\def\Rn{\Rx n}

\def\AG#1{\ro{A}(#1)}  
\def\AGR#1{\AG{#1,\R}} 
\def\SA#1{\ro{SA}(#1)} 
\def\SAR#1{\SA{#1,\R}} 
\def\GL#1{\ro{GL}(#1)} 
\def\GLR#1{\GL{#1,\R}} 
\def\SL#1{\ro{SL}(#1)} 
\def\SLR#1{\SL{#1,\R}}  

\def\semidirect{\ltimes}

\numberwithin{equation}{section}
\newcommand\blfootnote[1]{%
 \begingroup
 \renewcommand\thefootnote{}\footnote{#1}%
 \addtocounter{footnote}{-1}%
 \endgroup
 }

\begin{document}

\title[Hypersurface flow in centro-affine geometry]{An eternal hypersurface flow arising in centro-affine geometry}
   	\author[X. J. Jiang]{Xinjie Jiang}
    \address{ Xinjie Jiang\newline\indent
     Department of Mathematics, Northeastern University, Shenyang, 110819, P.R. China}
    \email{serge0912@icloud.com}

     \author[C. Z. Qu]{Changzheng Qu}
    \address{Changzheng Qu\newline\indent
     School of Mathematics and Statistics, Ningbo University, Ningbo, 315211, P.R. China}
    \email{quchangzheng@nbu.edu.cn}

    \author[Y. Yang]{Yun Yang${}^*$}\blfootnote{${}^*$~Corresponding author: yangyun@mail.neu.edu.cn}
    \address{ Yun Yang\newline\indent
     Department of Mathematics, Northeastern University, Shenyang, 110819, P.R. China}
    \email[Corresponding author]{yangyun@mail.neu.edu.cn}

\begin{abstract}
    In this paper, the existence and uniqueness for a specific centro-affine invariant hypersurface
    flow in $\R^{n+1}$ are studied, and the corresponding evolutionary processes in both centro-affine and Euclidean settings are explored. It turns out that the flow exhibits similar properties as the standard heat flow. In addition, the long time existence of the flow is investigated, which asserts that the hypersurface governed by the flow  converges asymptotically  toward an ellipsoid via systematically investigating evolutions of the centro-affine invariants. Furthermore,  the classification of the eternal solutions for the flow is provided.

\end{abstract}




\maketitle
\section{Introduction}
In this paper, we consider an invariant hypersurface flow
\begin{equation}\label{mainflow}
	\frac{\partial X}{\partial t}=\left(-\frac{1}{2n}\log\psi+\lambda\right)X,\quad X(\cdot,0)=X_0
\end{equation}
in centro-affine geometry, where $ \psi $ is the Tchebychev function (see \eqref{Tfun} for a precise description of the Tchebychev function), $ \lambda $ is arbitrary constant, and $X_0$ is a smooth, closed, uniformly convex hypersurface in $\R^{n+1}$.

It has been known that geometric flows in various geometries have been studied extensively. For the invariant submanifold flow (cf. \cite{olv-isf}), a fundamental issue is to determine the induced evolution of the geometric invariants characterizing the submanifold $S$. Among all invariant submanifold flows, the easiest thing to think of is the group invariant geometric heat flow. As usual, group invariant geometric heat flow means that a geometric invariant $F$ on the submanifolds evolves according to the partial differential equation
$ \frac{\partial}{\partial t}F=\Delta_g F$ with initial condition $F(\cdot, 0)=F_0:\mathcal{M}\rightarrow \mathcal{N}$, where $\Delta_g$ is the Laplace-Beltrami operator on $(\mathcal{M},g)$, and $g$ denotes the group invariant metric on $\mathcal{M}$.
The curve shortening flow (CSF) \cite{gh}, the affine curve shortening flow (ACSF) \cite{ast,st}, the mean curvature flow (MCF) \cite{Huisken0}, the affine normal flow (ANF) \cite{and0} are commonly geometric heat flows.
The Ricci flow (RF) is the analogue of the heat equation on a Riemannian manifold \cite{HamiltonRF0}, which was applied successfully by Perelman to solve the Poincar\'e conjecture completely.
Generally, the geometric heat flows induce the second-order nonlinear parabolic equations. For example, CSF, ACSF, MCF, ANF and RF.
Interestingly, the geometric heat flow in centro-affine geometry leads to the well-known inviscid Burgers’ equation \cite{oqy}, and the geometric heat flow in general-affine geometry yields the fourth-order nonlinear parabolic equation.
In \cite{jyy, qy}, an invariant second-order nonlinear parabolic equation with respect to centro-affine transformation group in $\R^2$ was investigated, which has some analogous properties to the group-invariant heat flows.
It is worth noting that the flow \eqref{mainflow} generates the  second-order nonlinear parabolic equations with respect to centro-affine invariants (see Section \ref{sec-UniformE}), and also  possesses certain properties similar to the standard heat flow (see Section \ref{sec-ProofA}).

The study of affine differential geometry relies on  the Lie group $\AGR n = \GLR n \semidirect \Rn$ consisting of affine transformations $x \longmapsto Ax+b$, $A\in \GLR n$, $b\in \Rn$ acting on $x \in \Rn$
(see Nomizu and Sasaki \cite{ns} and Simon \cite{sim} for details), and correspondingly equi-affine geometry is restricted to the subgroup $\SAR n = \SLR n \semidirect \Rn$ of volume-preserving affine transformations. Centro-affine differential geometry refers to the subgroup of the affine transformation group that keeps the origin fixed,  which is closely related to the geometry induced
by the general linear group $x \longmapsto Ax$, $A\in \GLR n$, $x \in \Rn$.
Furthermore, centro-equi-affine differential geometry arises in connection with the subgroup $\SLR n$ of volume-preserving linear transformations. In the centro-affine differential geometry, the ambient
space $\R^{n+1}$ has a flat affine connection $D$ and the usual determinant function
 is regarded as a parallel volume element. Let $ \mathcal{M} $ be an $n$-dimensional smooth manifold and let $ X: \mathcal{M}\rightarrow\R^{n+1} $ be a hypersurface immersion such that $ X $ is transversal to $ X_*(T\mathcal{M}) $ at each point of $ \mathcal{M} $. Note that the Einstein summation convention is employed and the range of indices  is  $ 1\le i,j,k,\cdots\le n $.
 The centro-affine metric $g$ on $ T\mathcal{M} $ is defined through a symmetric bilinear form, that is
\begin{align*}
	g=-\frac{[X_*(E_1),X_*(E_2),\cdots,X_*(E_n),E_iE_j(X)]}{[X_*(E_1),X_*(E_2),\cdots,X_*(E_n),X]}\theta^i\otimes\theta^j=g_{ij}\theta^i\otimes\theta^j,
\end{align*}
where the bracket notation $ [\;\cdots\;] $ is used to denote the standard determinant in $ \R^{n+1} $, and $ \{E_i\} $ is a local basis for $ T\mathcal{M} $ with the dual basis $ \{\theta^i\} $ (see \cite{lls,lsz,wang} for more details).
For locally strongly convex hypersurfaces, this quadratic form is positive definite by appropriate choice of the orientation. From now on, we restrict to locally strongly convex hypersurfaces as in this case the centro-affine metric is a Riemannian metric.

The difference of the Levi-Civita connection $ \nabla $ of $ g $ and the induced connection $ \widehat{\nabla} $ is a $ (1, 2) $-tensor $ C $ on $ \mathcal{M} $ with the property that its associate cubic form $ C $, defined by
\begin{align*}
	C(u,v,w)=g(C(u,v),w), \quad u,v,w\in T\mathcal{M}
\end{align*}
is totally symmetric. The function $ J $ on $ \mathcal{M} $ given by
\begin{align*}
	J=\frac{1}{n(n-1)}|C|^2
\end{align*}
is called the Pick invariant, where $ | \cdot | $ denotes the norm of a tensor with respect to the centro-affine metric $ g $. The Tchebychev form and the Tchebychev vector field on $ \mathcal{M} $ are defined by $	T=\frac{1}{n}\text{trace}_g(C)$
and $g(T,v)=T(v), v\in T\mathcal{M}$, respectively. The centro-affine mean curvature is defined by
    \begin{equation*}
     H=\frac{1}{n}\mathrm{trace}_g(\nabla T)=\frac{1}{n}\mathrm{Div}~T,
    \end{equation*}
    where $\mathrm{Div}$ is the divergence operator on $\mathcal{M}$.
The positive function $\psi$ (called the Tchebychev function of $X$)
\begin{align}\label{Tfun}
\psi=\frac{\det(g_{ij})}{[X_*(E_1),X_*(E_2),\cdots,X_*(E_n),X]^2}
\end{align}
is independent of choice of the frame and is invariant under centro-equi-affine transformation (cf. \cite{lls,lsz}).
For the equi-affine support function $ \rho $ (Section 4.13 in \cite{sim}) and the Tchebychev vector filed of centro-affine geometry, the following equation holds \cite{lls}
\begin{align}\label{rel-support}
	T=-\frac{1}{2n}D\left(\log\psi\right)=\frac{n+2}{2n}D\left(\log\rho\right).
\end{align}
In particular, the relation $ T=0 $  characterizes proper affine spheres since $ \rho $ is a constant (cf. \cite{lls}).

With the centro-affine transformation $\bar{X}=AX$,  the flow \eqref{mainflow}
may be rewritten as
$$\displaystyle
\frac{\partial \bar{X}}{\partial t}=\left(-\frac{1}{2n}\log\bar{\psi}+\bar{\lambda}\right)\bar{X},\qquad \bar{X}(\cdot,0)=\bar{X}_0,
$$
where $A$ is a nonsingular $\displaystyle (n+1)\times(n+1)$ matrix, $\bar{\psi}$ is the Tchebychev function of $\bar{X}$ and $\displaystyle\bar{\lambda}=\lambda-\log(\det(A))^2/(2n)$
is still constant.
In addition, the constant $\lambda$ can be eliminated
by a time-rescaling $\tilde{X}=f(t)X$ with $\displaystyle f(t)=\exp\left(\frac{n\lambda}{n+1}\left(1-\exp (\frac{n+1}{n}t)\right)\right)$, that is,
$$\displaystyle
\frac{\partial \tilde{X}}{\partial t}=\left(-\frac{1}{2n}\log\tilde{\psi}\right)\tilde{X},\quad \tilde{X}(\cdot,0)=X_0,
$$
where $\tilde{\psi}$ is the Tchebychev function of $\tilde{X}$.
This indicates the flow \eqref{mainflow} is invariant under centro-affine transformation.
The existence and uniqueness are among the most fundamental and critical issues of the geometric flows. For this flow, the following theorem is one of main results.
\begin{thm}\label{Uin-Exi}
	Let $ X_0 $ be a smooth, closed, uniformly convex hypersurface in $ \R^{n+1} $ which encloses the origin. Then the flow \eqref{mainflow} has a unique smooth, uniformly convex solution $ X(\cdot,t) $ in $ [0,+\infty) $.
\end{thm}
Theorem \ref{Uin-Exi} will be proved in Section \ref{sec-ER} by an
approach similar to that used in \cite{Chou,LQR,jieu}, namely, by introducing the Euclidean support function of $X(\cdot,t)$ and reducing \eqref{mainflow} to a single parabolic equation of
Monge-Amp\`{e}re type on this support function.

	Over the past four decades, extensive research has been conducted on the asymptotic behavior of convex hypersurfaces in Euclidean and affine spaces based on invariant geometric flows. Prominent examples in Euclidean geometry include MCF and the Gauss curvature flow (GCF). 
	MCF was first investigated by Brakke \cite{Brakke}, in the context of geometric measure theory. In an intriguing paper \cite{Huisken0}, Huisken proved that,  for a smooth and uniformly convex initial hypersurface, MCF remains convex and contracts smoothly to a round point in finite time, ``round" meaning that after a suitable time-dependent rescaling the flow converges smoothly to a sphere. This property was extended to other geometric flows in which the speed is a homogeneous function of degree one with respect to the principal curvatures \cite{Andrew1,Chow1,Chow2}.
	GCF was introduced by Firey in his seminal paper \cite{Firey} as a model of the wearing process undergone by a pebble on a beach.
	In the special case of surfaces in $\R^3$, Andrews \cite{AndrewGCF} showed that the flow deforms a uniformly convex hypersurface into a round point. In the higher dimensional case, Tso \cite{TsoK} proved that the flow exists up to some maximal time, and as it approaches this maximum time, the volume enclosed by the hypersurface converges to zero. Recently, by combining a soliton convergence result \cite{AndrewsGuanNi,GuanNi} and a uniqueness result for the soliton \cite{BrendleGaCSF}, the asymptotic behavior of GCF in higher dimensions has been fully resolved. In fact, in \cite{AndrewsGuanNi,BrendleGaCSF}, they considered a more general flow where the speed is the power of the Gauss curvature, and GCF is one of its special cases.

	Unlike the spherical asymptotic behavior of flows in Euclidean space, invariant geometric flows in affine geometry often exhibit ellipsoidal asymptotic behavior due to the affine invariance of these flows. In \cite{st}, Sapiro and Tannenbaum studied ACSF, which is the equi-affine analogue of CSF, and demonstrated that a closed convex embedded curve evolving under this flow will shrink to an elliptical point. Andrews \cite{and0} later extended this conclusion to the higher dimensional case, namely, ANF. Additionally, because ACSF is not pointwise shortening the equi-affine length, in order to obtain a more natural equi-affine analogue of CSF, Andrews also studied an $L^2$ gradient flow \cite{and}, where the equi-affine length increases pointwise along this flow. He proved that the flow expands to infinity and approaches a homothetically expanding ellipse. For centro-equiaffine geometric flows, there are a series of achievements obtained by Ivaki and Stancu. See \cite{is,iva0,iva1} and references therein. Wo, Wang, and Qu \cite{wwq} studied the centro-equiaffine geometric heat flow, and they showed that the original symmetric solution of the flow shrinks to an elliptical point. The geometric heat flows in centro-affine geometry and more general Klein geometries were investigated in \cite{olv-isf,oqy,ost-1,yq}.

In this paper, by analysing the asymptotic behaviors of the cubic form and the Tchebychev vector field, we obtain the following theorem  (its proof will be provided in Section \ref{sec-ProofA}).
\begin{thm}\label{mainthmA}
	Assume $ X(\cdot,t) $ is a solution of the flow \eqref{mainflow} in a maximal interval $ [0,+\infty) $, where $ {X}_0 $ is a smooth, closed, uniformly convex hypersurface. Then the Tchebychev vector field  $ T $ of $X(\cdot,t)$ converges smoothly to zero as $ t\rightarrow +\infty $,
	that is, $ X(\cdot,t) $ converges to an ellipsoid centred at origin as $ t\rightarrow +\infty. $
\end{thm}

Ancient solutions play an essential role in the singularity analysis of geometric flows. In order to study the geometric behavior at singularities, one needs to classify such solutions. For embedded convex compact ancient solutions of CSF, Daskalopoulos, Hamilton, and Sesum \cite{dphrsn} verified that there are only two possibilities: the shrinking circles, which sweep the whole space $ \R^2 $, or Angenent ovals, which can decompose into two translating solutions of the flow. Later, Daskalopoulos, Sesum and  Angenent obtained some results on  ancient solutions of MCF  \cite{asdpsn,asdpsn1}. They verified that, under certain conditions, compact non-collapsed ancient solutions of the flow exhibit unique asymptotics as $ t\rightarrow -\infty$, which can be described by solutions constructed in \cite{WhiteB,HRHO}. In addition, Huisken and Sinestrari \cite{HuiskenSC} discussed  the conditions that ensure closed convex ancient solutions to MCF can only be shrinking spheres. For noncompact ancient solutions of MCF, Brendle and Choi \cite{BrendleCK1} provided a complete classification when certain constraints are imposed on the solutions. More specifically, they showed that any noncompact and complete ancient solutions of MCF, which are strictly convex, uniformly two-convex, and non-collapsed, must be a Bowl soliton, up to scaling and ambient isometries. For the classification of ancient solutions to RF, there has been an increasing amount of research, see \cite{Brendle5,Brendle2,Brendle3,Brendle4, Brendle1,dphrsn1} and the references therein.
	
	There are also some classification results about ancient solutions to affine geometric flows. In \cite{lt}, Loftin and Tsui asserted that the only compact convex ancient solutions to ANF in $\R^{n+1}, n\ge2$, are contracting ellipsoids. However, their approach can not be used to ACSF, as it relies on a fact that the cubic form vanishes implies only the hypersurface is a hyperquadric in higher dimensions. The case of the plane (ACSF) was addressed by Chen \cite{chen}, who proved that the only compact convex ancient solutions to ACSF are shrinking ellipses. Similar results were also obtained by Ivaki \cite{IvakiACSF} using a different method. Additionally, Ivaki classified ancient solutions of the planar $p$-centro-equiaffine normal flows \cite{IvakiCPANF} using a similar approach,  and demonstrated that the only compact, origin-symmetric, strictly convex ancient solutions of the flows are contracting origin-centered ellipses.
	Indeed we have the following result (its proof will be provided in Section \ref{sec-EternalSolutions}).
\begin{thm}\label{mainthmB}
	 The only smooth closed strictly convex eternal solutions of the flow  \eqref{mainflow} with assumption that $ |T|^2 $ is uniformly bounded on $ (-\infty,+\infty)  $ are  origin-centered ellipsoids.
\end{thm}
This paper is organized as follows. In Section \ref{sec-Basic}, local expressions of centro-affine geometric quantities are given, and the centro-affine invariant hypersurface flow \eqref{flow}, together with evolution equations for some geometric quantities, are presented. In Section \ref{sec-ER}, we consider the flow in Euclidean setting and reduce it to a scalar parabolic equation of Monge-Amp\`ere type, via the support function. We establish the a priori estimates, which imply the long time existence and regularity of the flow \eqref{euclidean}. Some results about this flow in Euclidean setting are obtained in Section \ref{sec-Euclidean Results}. Uniform bounded estimates for some centro-affine invariants are derived in Section \ref{sec-UniformE}. Section \ref{sec-ProofA} is dedicated to the proof of the Theorem \ref{mainthmA}. The classification of eternal solutions for the flow is provided in Section \ref{sec-EternalSolutions} by demonstrating that the backward limit of these solutions can only be ellipsoids.

\section{Evolution equations of the hypersurface flow}\label{sec-Basic}
We now utilize the natural basis, allowing us to represent the Gauss equation of $ X $ in centro-affine geometry as follows
\begin{align*}
	\frac{\partial^2X}{\partial u^i\partial u^j}=\widehat{\Gamma}_{ij}^k\frac{\partial X}{\partial u^k}-g_{ij}X,
\end{align*}
where $  \widehat{\Gamma}_{ij}^k $ are the components of the induced connection $ \widehat{\nabla} $. Let $ \Gamma_{ij}^k $ be the Christoffel symbols of the centro-affine metric $ g_{ij} $. Thus, we have the following notations:
\begin{align*}
	&C_{ij}^k=C_{ji}^k=\widehat{\Gamma}_{ij}^k-\Gamma_{ij}^k, \qquad C_{ijk}=C_{ij}^lg_{lk},\\
	&T^k=\frac{1}{n}g^{ij}C_{ij}^k, \qquad T_i=T^kg_{ik}=\frac{1}{n}C_{ki}^k,\\
	&J=\frac{1}{n(n-1)}g^{ij}C_{im}^lC_{lj}^m,
\end{align*}
where $ g^{ij} $ is the inverse matrix of $ g_{ij} $. Unless otherwise noted, we raise and lower indices using the centro-affine metric $ g_{ij} $ and $ g^{ij} $.
In particular, the Riemannian curvature tensor is given by
\begin{align*}
	R_{ijk}^l=\frac{\partial\Gamma_{jk}^l}{\partial u^i}-\frac{\partial\Gamma_{ik}^l}{\partial u^j}+\Gamma_{jk}^p\Gamma_{pi}^l-\Gamma_{ik}^p\Gamma_{jp}^l, \qquad R_{ijkm}=R_{ijk}^lg_{lm},
\end{align*}
and the normalized scalar curvature $\chi$ with respect to the centro-affine metric $ g $ can be expressed by
\begin{align*}
	&\chi=J-\frac{n}{n-1}|T|^2+1.
\end{align*}
The following symmetries are obtained in \cite{wang}
\begin{align*}
	&C_{ijk;l}=C_{ijl;k}, \qquad T_{i;j}=T_{j;i},
\end{align*}
where ``;" denotes the covariant derivatives with respect to the centro-affine metric $ g $.

Assume $X_0$ is a smooth, closed, uniformly convex hypersurface in $\R^{n+1}$. Then let us investigate the following geometric flow
\begin{eqnarray}\label{flow}
\begin{aligned}
	&\frac{\partial X}{\partial t}=\frac{1}{2}T^kX_k+\left(-\frac{1}{2n}\log\psi+\lambda\right)X,\\
   & X(\cdot,0)=X_0,
\end{aligned}
\end{eqnarray}
which is, up to a tangential diffeomorphism, equivalent to the flow \eqref{mainflow}.
From \eqref{rel-support}, it is readily apparent that the proper affine hypersphere is the centro-affine stationary solution of the flow \eqref{flow}.

Next, we employ the general formula in \cite{yq} to identify  the evolution equations for some geometric quantities under the flow \eqref{flow}. For the reader's convenience, some calculation details will be presented.
\begin{prop}\label{evo}
	Assume that the hypersurface $ X(\cdot,t) $ is evolving under the flow \eqref{flow}. Then the following evolution equations hold:
	\begin{align*}
		&(1)\quad\frac{\partial}{\partial t}g_{ij}=T_pC^p_{ij},\\
		&(2)\quad\frac{\partial}{\partial t}g^{iq}=-T_pC^{piq},\\			
		&(3)\quad\frac{\partial}{\partial t}C_{ij}^k=\frac{1}{2}(T_{;ij}^k+T^k_{;ji}-g^{kl}T_{i;jl})+T_i\delta_j^k+T_j\delta_i^k,\\
		&(4)\quad\frac{\partial}{\partial t}C_{ijk}=\frac{1}{2}T_{i;jk}+\frac{1}{2}T_lC_{ik}^pC_{pj}^l+\frac{1}{2}T_lC_{jk}^pC_{ip}^l+\frac{1}{2}g_{ik}T_j+\frac{1}{2}g_{jk}T_i+g_{ij}T_k,\\
		&(5)\quad\frac{\partial}{\partial t}T_i=\left(1+\frac{1}{n}\right)T_i+\frac{1}{2}H_i.
	\end{align*}
\end{prop}
\begin{proof}
	Using the general formula from \cite{yq}, it is straightforward to derive the first two equations and the last one. For the third equation, we compute
	\begin{align*}
		\frac{\partial}{\partial t}C_{ij}^k=&\frac{1}{2}(T_{;j}^lC_{li}^k+T_{;i}^lC_{lj}^k)+(T_j\delta_i^k+T_i\delta_j^k)+\frac{1}{2}T^lC_{li;j}^k-\frac{1}{2}T_{;l}^kC_{ij}^l\\
		&-\frac{g^{kl}}{2}(T_pC_{ij;l}^p-T_{l;ij}-T_{j;li}+T_{i;jl}+T_{p;j}C_{li}^p+T_{p;i}C_{jl}^p-T_{p;l}C_{ij}^p)\\
		=&\frac{1}{2}(T_{;j}^lC_{li}^k+T_{;i}^lC_{lj}^k)+(T_j\delta_i^k+T_i\delta_j^k)+\frac{1}{2}T^lC_{li;j}^k-\frac{1}{2}T_{;l}^kC_{ij}^l\\
		&-\frac{1}{2}(T_pC_{i;j}^{pk}-T^k_{;ij}-T^k_{;ji}+g^{kl}T_{i;jl}+T_{p;j}C_{i}^{kp}+T_{p;i}C_j^{pk}-T^k_{;p}C_{ij}^p)\\
		=&\frac{1}{2}(T_{;ij}^k+T^k_{;ji}-g^{kl}T_{i;jl})+T_i\delta_j^k+T_j\delta_i^k.
	\end{align*}
Then
\begin{align*}
	\frac{\partial}{\partial t}C_{ijk}=&\;\frac{\partial}{\partial t}(C_{ij}^pg_{pk})=g_{pk}\left(\frac{\partial}{\partial t}C_{ij}^p\right)+C_{ij}^p\left(\frac{\partial}{\partial t}g_{pk}\right)\\
	=&\;g_{pk}\left[\frac{1}{2}(T_{;ij}^p+T^p_{;ji}-g^{pl}T_{i;jl})+T_i\delta_j^p+T_j\delta_i^p\right]+C_{ij}^pT_lC_{pk}^l\\
	=&\;\frac{1}{2}(T_{k;ij}+T_{k;ji}-T_{i;jk})+T_lC_{ij}^pC_{pk}^l+g_{ik}T_j+g_{jk}T_i\\
	=&\;\frac{1}{2}T_{i;jk}+\frac{1}{2}T_lC_{ik}^pC_{pj}^l+\frac{1}{2}T_lC_{jk}^pC_{ip}^l+\frac{1}{2}g_{ik}T_j+\frac{1}{2}g_{jk}T_i+g_{ij}T_k,
\end{align*}
where the last step follows from the symmetries of the cubic form $C$.
\end{proof}
\section{Existence and regularity results}\label{sec-ER}
 In this section, we discuss the existence and smoothness of solutions to the flow \eqref{mainflow}, the Euclidean support function will be employed to reduce this flow to  a parabolic equation of Monge-Amp\`{e}re type. Let us briefly review some basic facts about the Euclidean support function. Details can be found in \cite{jieu}.
\subsection{The Euclidean support function}
Let $ \mathcal{M} $ be a smooth, closed, uniformly convex hypersurface in $ \R^{n+1} $.
Suppose that $ \mathcal{M} $ is given as an embedding of $ \mathbb{S}^n $ via the inverse Gauss map $ X: \mathbb{S}^n\rightarrow \mathcal{M}\subset\R^{n+1} $. Without loss of generality, we may assume that $ \mathcal{M} $ encloses the origin. The support function $ s $ of $ \mathcal{M} $, which gives the perpendicular distance from
the origin to the tangent plane at $ X(x) $, is defined by
\begin{align*}
	s(x)=\left<x,X(x)\right>
\end{align*}
for each $ x $ in $ \mathbb{S}^n, $
where $ \left<\cdot,\cdot\right> $ denotes the standard inner product of $ \R^{n+1} $. We extend $ s $ to be a homogeneous function of degrees one on $ \R^{n+1}\backslash\{0\} $
\begin{align*}
	s(x)=\sup\limits_{y\in \mathcal{M}}\left<x,y\right>
\end{align*}
for all $ x\in \R^{n+1} $.
 $ s $ is convex because it is a supremum of linear functions.
 Moreover, $s$ is smooth due to the smoothness of $X$.

Conversely, if $ s $ is a convex function which is smooth and homogeneous of degree one on $ \R^{n+1}\backslash\{0\} $, then it can be shown that $ s $ is the support function of a unique convex hypersurface $ \mathcal{M}=\partial\Omega $, where $ \Omega $ is the convex body
\begin{align*}
	\Omega=\bigcap\limits_{x\in \mathbb{S}^n}\{y\in\R^{n+1}: \left<x,y\right>\le s(x)\}.
\end{align*}
Alternatively, the support function $ s $ can be used to define a canonical embedding $ \bar{X} $ of $ \mathbb{S}^n $ with image equal to $ \mathcal{M} $ (see e.g. \cite{TsoK}):
\begin{align}\label{embedding}
	\bar{X}(x)=Ds(x)=s(x)x+\bar{\nabla}s(x),
\end{align}
where $ D=(D_1,\cdots ,D_{n+1}) $ is the gradient on $\mathbb{R}^{n+1}$ and $ \bar{\nabla} $ is the Levi-Civita connection on  $ \mathbb{S}^n $ coming from the standard metric $ \bar{g} $.
This has the property that the outward normal to $ \mathcal{M} $ at the point $ \bar{X}(x) $ is equal to $ x $, for each $ x\in\mathbb{S}^n $. In fact, $ \bar{X} $ is nothing but the inverse of the Gauss map.

Geometric quantities of $ \mathcal{M} $ can now be expressed in terms of $ s $. Through a direct computation, one can determine that the principal radii of curvature of $ \mathcal{M} $ at $ X(x)  $ are precisely the eigenvalues of the matrix $ (b_{ij})=(\bar{\nabla}^2_{ij}s+s\delta_{ij}) $, where $ \{e_1,\cdots,e_{n} \}$ is an orthonormal frame fields on $ \mathbb{S}^n $, $ \bar{\nabla} $ is the covariant derivative with respect to $ e_i $. Notably, the Gauss curvature at $ X(x)$ is given by
\begin{align*}
	K=\frac 1{\det(\bar{\nabla}^2s+sI)}.
\end{align*}
Note that $ s $ is considered as a homogeneous function over $ \R^{n+1}\backslash\{0\} $, the principal radii of curvature of $ \mathcal{M} $ are also equal to the non-zero eigenvalues of the Hessian matrix $ (D^2_{ij}s)_{i,j=1,\cdots,n+1} $.
\subsection{The flow in Euclidean setting}
By \eqref{rel-support} and an appropriate time-rescaling, the flow \eqref{mainflow}
is equivalent to
\begin{equation}\label{mainflow1}
	\frac{\partial X}{\partial t}=\left(\frac{n+2}{2n}\log\rho\right)X,\qquad X(\cdot,0)=X_0.
\end{equation}
   Note that the Euclidean support function and the equi-affine support function satisfy $ \rho=sK^{-\frac{1}{n+2}} $ (see e.g. \cite{ns,zgx}), where $ K $ is the Gauss curvature. Thus this flow in Euclidean setting may be rewritten as
\begin{equation}\label{euclidean}
	\frac{\partial X}{\partial t}=\frac{s}{2n}\log\left(\frac{s^{n+2}}{K}\right)x, \qquad X(\cdot,0)=X_0.
\end{equation}
In fact, the initial value problem \eqref{euclidean} can be transformed into an initial value problem concerning the support function $s$ (cf. \cite{jieu}):
\begin{equation}\label{sevo}
	\left\{
	\begin{aligned}
		&\frac{\partial s}{\partial t}=\frac{s}{2n}\log(s^{n+2}\det(\bar{\nabla}^2 s+sI)),\\
		&s(x,0)=s_0(x),
	\end{aligned}
\right.
\end{equation}
where $ x\in \mathbb{S}^n $ and $ s_0 $ is the support function for $ X_0$.
\subsection{A priori estimates}
In this subsection we will always assume that $ s\in C^{\infty}(\mathbb{S}^n\times[0,T]) $ is a positive and uniformly convex solution to \eqref{sevo}.
Let $  R(t) $ and $ r(t) $ be the outer and inner radii of the hypersurface $\mathcal{M}_t=X(\mathbb{S}^n,t) $ determined by $ s(\cdot,t) $ respectively. Suppose that
\begin{equation*}
	{	R_0=\sup\{R(t): t\in [0,T]\}}
\end{equation*}
and
\begin{equation}\label{r_0}
	{	r_0=\inf\{r(t): t\in [0,T]\}.}
\end{equation}
Our goal is to estimate the principal radii of curvatures of $ \mathcal{M}_t $ from below and above in terms of $ r_0, R_0 $, and initial data. We first establish the following time-dependent $ C^0 $ estimate for the solution to the flow \eqref{sevo}.
\begin{lem}\label{C0}
	Let $ s(\cdot,t) $, $ t\in [0,T] $, be a solution to the flow \eqref{sevo}. Then
	\begin{align*}
		s(\cdot,t)\le R_0\le \max\{(\max_{\mathbb{S}^n}s_0)^{\exp(\frac{n+1}{n}T)},1\}
	\end{align*}
	and
	\begin{align*}
		s(\cdot,t)\ge r_0\ge \min\{(\min_{\mathbb{S}^n}s_0)^{\exp(\frac{n+1}{n}T)},1\}.
	\end{align*}
\end{lem}
\begin{proof}
	Let $ s_{\max}(t)=\max_{x\in\mathbb{S}^n}s(x,t)$, which is also denoted as $s(x_t,t) $. Thus for a fixed time $ t $, at the point $ x_t $, thanks to maximality, one may find $ \bar{\nabla}s(x_t,t)=0 $, and $ \bar{\nabla}^2s(x_t,t) $ is negative semi-definite. According to
	\begin{align*}
		\det(\bar{\nabla}^2s+sI)(x_t,t)\le\det(\bar{\nabla}^2s+sI-\bar{\nabla}^2s)=s^n(x_t,t),
	\end{align*}
	\eqref{sevo} generates
	\begin{align*}
		\frac{ds_{\max}}{dt}\le\frac{n+1}{n}s_{\max}\log(s_{\max}).
	\end{align*}
	Since $ s_{\max}(t)\leq1 $ supports the desired result, we assume $ s_{\max}(t)>1 $. Solving the above inequality yields
	\begin{align*}
		s_{\max}(t)\le s_{\max}(0)^{\exp(\frac{n+1}{n}t)}.
	\end{align*}
	Hence we conclude that
	\begin{align*}
		s(\cdot,t)\le R_0\le \max\{s_{\max}(0)^{\exp(\frac{n+1}{n}T)},1\}.
	\end{align*}
	By a similar manipulation, the inequality
	\begin{align*}
			s(\cdot,t)\ge r_0\ge \min\{s_{\min}(0)^{\exp(\frac{n+1}{n}T)},1\}
	\end{align*}
 can be proved.
\end{proof}
For a convex hypersurface $ \mathcal{M} $, the gradient estimate is a direct consequence of the $ L^\infty $-norm estimate.
\begin{lem}\label{C1}
		Let $ s(\cdot, t), t\in [0,T], $ be a solution to the flow \eqref{sevo}. Then
	\begin{equation*}
		|\bar{\nabla} s(\cdot, t)|\le \max_{{\mathbb{S}}^n\times[0,T]}s, \qquad \forall t \in [0,T].
	\end{equation*}
\end{lem}
\begin{proof}
	This is due to convexity.
\end{proof}
In order to derive the a priori estimates for the higher derivatives of $ s $,  it is advantageous to express \eqref{sevo} locally in Euclidean space. Consider the restriction of $ s(x,t) $ to the hyperplane $  x_{n+1} = -1 $, denoted as $ u(y,t) = s(y,-1,t) $. Then $ u $ is convex in $ \R^n $,  and
\begin{gather*}
	\det\left(D^2u\right)(y,t)=(1+|y|^2)^{-\frac{n+2}{2}}\det(\bar{\nabla}^2s+sI)(x,t),\\
	\frac{\partial u}{\partial t}(y,t)=\sqrt{1+|y|^2}\frac{\partial s}{\partial t}(x,t),
\end{gather*}
for $ x=(y,-1)/\sqrt{1+|y|^2} $. Therefore,
\begin{equation}\label{eqR}
	\frac{\partial u}{\partial t}(y,t)=\frac{1}{2n}u\log(\det D^2u)+\frac{n+2}{2n}u\log u, \qquad y\in\R^n.
\end{equation}
Now let us begin with estimating upper and lower bounds for the principal radii of curvatures of $ X(\cdot,t) $ based on $ r_0^{-1}, R_0 $, and the initial data. We first provide a uniform upper bound on the Gauss curvature of $ X(\cdot,t) $, which is equivalent to a uniformly positive lower bound on $ \det(\bar{\nabla}^2s+sI) $.
\begin{lem}\label{gauss}
		Let $ s(\cdot, t), t\in [0,T] $, be a solution to the flow \eqref{sevo}. Then there is a constant $ L > 0 $, depending only on $ n $, $ r_0 $, $ R_0 $, and initial hypersurface $ \mathcal{M}_0 $, such that
	\begin{equation*}
					\det(\bar{\nabla}^2 s+sI)\ge L,\qquad \forall (x,t)\in \mathbb{S}^n\times [0,T].
	\end{equation*}
\end{lem}
\begin{proof}
	Let us introduce the auxiliary function
	\begin{equation*}
		\Psi(x,t)=-\frac{s_t(x,t)}{s(x,t)-\frac{1}{2}r_0},
	\end{equation*}
	where the positive constant $ r_ 0$ is defined in \eqref{r_0}.
	Suppose that the supremum of $ \Psi(x,t) $ is attained at the south pole $ x=(0,\cdots,0,-1)$ and $ \bar{t}>0. $
	Let $ u $  indicate the restriction of $ s $ on the hyperplane $ x_{n+1}=-1 $. Then
	\begin{equation*}
		\psi(y,t)=-\frac{u_t(y,t)}{u(y,t)-\frac{1}{2}r_0\sqrt{1+|y|^2}}
	\end{equation*}
	attains its supremum at $ (0,\bar{t}) $. Hence at $ (0,\bar{t}) $ we have, for each $ k $,
	\begin{align*}
		0\le&\psi_t=-\frac{u_{tt}}{u-\frac{1}{2}r_0}+\frac{u_t^2}{(u-\frac{1}{2}r_0)^2},\\
		0=&\psi_k=-\frac{u_{kt}}{u-\frac{1}{2}r_0}+\frac{u_tu_k}{(u-\frac{1}{2}r_0)^2},
	\end{align*}
	and
	\begin{align*}
		0\ge \psi_{kk}=-\frac{u_{kkt}}{u-\frac{1}{2}r_0}+\frac{u_t(u_{kk}-\frac{1}{2}r_0)}{(u-\frac{1}{2}r_0)^2}.
	\end{align*}
	On the other hand, differentiating \eqref{eqR} produces
	\begin{align*}
		u_{tt}=&\frac{1}{2n}u_t\log(\det D^2u)+\frac{1}{2n}uu^{ij}u_{ijt}+\frac{n+2}{2n}u_t\log u+\frac{n+2}{2n}u_t\\
		=&\frac{u_t^2}{u}+\frac{1}{2n}uu^{ij}u_{ijt}+\frac{n+2}{2n}u_t.
	\end{align*}
	Rotating the axes so that $ \{u^{ij}\} $ is diagonal at $ (0,\bar{t}) $. Then
	\begin{align*}
		0\ge& \frac{1}{2n}uu^{kk}\psi_{kk}-\psi_t\\		=&\frac{1}{2n}uu^{kk}\left(-\frac{u_{kkt}}{u-\frac{1}{2}r_0}+\frac{u_t(u_{kk}-\frac{1}{2}r_0)}{(u-\frac{1}{2}r_0)^2}\right)+\frac{u_{tt}}{u-\frac{1}{2}r_0}-\frac{u_t^2}{(u-\frac{1}{2}r_0)^2}\\
		=&\frac{uu_t(n-\frac{1}{2}r_0\sum_{k=1}^n u^{kk})}{2n(u-\frac{1}{2}r_0)^2}+\frac{u_t^2}{u(u-\frac{1}{2}r_0)}+\frac{(n+2)u_t}{2n(u-\frac{1}{2}r_0)}-\frac{u_t^2}{(u-\frac{1}{2}r_0)^2}.
	\end{align*}
	Since $u_t\geq-1$ yields the desired result by \eqref{eqR}, we can assume  $ u_t<-1 $ at $ (0,\bar{t}) $. Multiplying the two sides of the above inequality by $ u_t^{-1}$, we obtain
	\begin{align*}
		0\le\frac{u(n-\frac{1}{2}r_0\sum_{k=1}^n u^{kk})}{2n(u-\frac{1}{2}r_0)^2}+\frac{u_t}{u(u-\frac{1}{2}r_0)}+\frac{n+2}{2n(u-\frac{1}{2}r_0)}-\frac{u_t}{(u-\frac{1}{2}r_0)^2}.
	\end{align*}
	Notice that $\sum_{k=1}^n u^{kk}\ge n(\det D^2u)^{-\frac{1}{n}} $. Therefore, in view of \eqref{eqR}, we find,
	\begin{align*}
		\frac{n}{2}r_0u^2(\det D^2u)^{-\frac{1}{n}}&\le \frac{n}{2}r_0u\log\left((\det D^2u)^{-\frac{1}{n}}\right)+\frac{n+2}{2}r_0u\log \left(u^{-1}\right)\\
		&\qquad+nu^2+(n+2)\left(u-\frac{1}{2}r_0\right)u.
	\end{align*}
	Applying Lemma \ref{C0} leads to the estimate
	\begin{equation*}
		(\det D^2u)^{-\frac{1}{n}}\le L\left(1+\log\left((\det D^2u)^{-\frac{1}{n}}\right)\right),
	\end{equation*}
	where $ L $ is an appropriate positive constant that depends only on  $ n $, $ r_0 $ and $R_0 $.
	The result easily follows.
\end{proof}
Now we can prove that the principal curvature radii of $ X(\cdot,t) $ are bounded by positive constants from both above and below.
\begin{lem}\label{2destimate}
	Let $ s(\cdot, t), t\in [0,T] $ be a solution to the flow \eqref{sevo}.
	Then, for all $ t\in [0,T] $ and $ x\in \mathbb{S}^n $, one has
	\begin{eqnarray}\label{C2ie}
		\frac{1}{L}I\le (\bar{\nabla}^2s+sI)\le LI,
	\end{eqnarray}
	where $ L $ is a positive constant depending only on  $ r_0, R_0 $ and the initial values.
\end{lem}
\begin{proof}
	We start by considering the following auxiliary function
	\begin{eqnarray*}
		\Phi(x,t)=\sqrt{s_{\xi\xi}(x,t)}+|\bar{\nabla} s(x,t)|^2+|s(x,t)|^2.
	\end{eqnarray*}
	Suppose that the supremum
	\begin{eqnarray*}
		{	\sup\{\Phi(x,t): (x,t)\in \mathbb{S}^n\times[0,T]}, \text{$ \xi $ tangential to $ \mathbb{S}^n $} , |\xi|=1\}
	\end{eqnarray*}
	is attained at the south pole $ x=(0,\cdots, 0, -1) $ at $ t=\bar{t}>0 $ and in the direction $ \xi=e_1 $ . For any $ x $ on the south hemisphere, let
	\begin{eqnarray*}
		\xi(x)=\left(\sqrt{1-x_1^2},-\frac{x_1x_2}{\sqrt{1-x_1^2}},\cdots,-\frac{x_1x_{n+1}}{\sqrt{1-x_1^2}}\right).
	\end{eqnarray*}
	Let $ u $ be the restriction of $ s $ on $ x_{n+1}=-1 $. For those points on $\mathcal{M}_t$ whose outer normals lie in the south hemisphere ($x_{n+1}<0$), their coordinates can be computed by (see e.g. \cite{TsoK})
	\begin{align*}
		X=\left(\frac{\partial u}{\partial y_1},\cdots,\frac{\partial u}{\partial y_n},\sum_{i=1}^{n}y_i\frac{\partial u}{\partial y_i}-u\right).
	\end{align*}
	  The homogeneity of $ s $ gives
	\begin{eqnarray*}
		s_{\xi\xi}(x,t)=u_{11}(y,t)\frac{(1+y_1^2+\cdots+y_n^2)^\frac{3}{2}}{1+y_2^2+\cdots+y_n^2},
	\end{eqnarray*}
	where $ y=-(x_1,\cdots,x_n)/x_{n+1} $ in $ \R^n $.
	Thus, in view of \eqref{embedding}, the function
	\begin{eqnarray*}
		\phi(y,t)=\sqrt{u_{11}\frac{(1+y_1^2+\cdots+y_n^2)^\frac{3}{2}}{1+y_2^2+\cdots+y_n^2}}+\sum_{i=1}^{n}u_i^2+\left(\sum_{i=1}^{n}y_iu_i-u\right)^2
	\end{eqnarray*}
	attains its maximum at $ (y,t)=(0,\bar{t}) $. Without loss of generality we may further assume that the Hessian of $ u $ at $ (0,\bar{t}) $ is diagonal. Hence at $ (0,\bar{t}) $ we have, for each $ k $,
	\begin{align}
		0\le \phi_t&=\frac{1}{2}(u_{11})^{-\frac{1}{2}}u_{11t}+2\sum_{i=1}^nu_iu_{it}+2uu_t,\\
		0=\phi_k&=\frac{1}{2}(u_{11})^{-\frac{1}{2}}u_{11k}+2\sum_{i=1}^nu_iu_{ik}\label{phik}
	\end{align}
	and
	\begin{align*}
		0\ge\phi_{kk}=&\frac{1}{2}(u_{11})^{-\frac{1}{2}}(u_{11kk}+\tau_ku_{11})-\frac{1}{4}(u_{11})^{-\frac{3}{2}}(u_{11k})^2\\
		&\quad+2\left(\sum_{i=1}^nu_{ik}^2+\sum_{i=1}^nu_iu_{ikk}\right)-2uu_{kk},
	\end{align*}
	where $ \tau_k=3 $ if $ k>1 $ and $ \tau_1=1 $.
	On the other hand, differentiating \eqref{eqR} yields
	\begin{align}
	\label{ukt}	u_{kt}=&\frac{1}{2n}uu^{ii}u_{iik}+\frac{1}{2n}u_k\log(\det D^2u)+\frac{n+2}{2n}u_k\log u+\frac{n+2}{2n}u_k,\\\notag
		u_{kkt}=&\frac{1}{2n}uu^{ii}u_{iikk}-\frac{1}{2n}uu^{ii}u^{jj}(u_{ijk})^2+\frac{1}{n}u_ku^{ii}u_{iik}+\frac{n+2}{2n}u_{kk}\\\label{ukkt}
		&\quad +\frac{1}{2n}u_{kk}\log(\det D^2u)+\frac{n+2}{2n}u_{kk}\log u+\frac{n+2}{2n}\frac{u_k^2}{u},
	\end{align}
	where $ \{u^{ij}\} $ is the inverse matrix of $ \{u_{ij}\} $. Hence at $ (0,\bar{t}) $ we have
	\begin{align*}
		0\le& \phi_t-\frac{1}{2n}uu^{kk}\phi_{kk}\\
		=&\frac{1}{2}(u_{11})^{-\frac{1}{2}}\left(u_{11t}-\frac{1}{2n}uu^{kk}u_{11kk}\right)+2\left(\sum_{i=1}^nu_iu_{it}-\frac{1}{2n}u\sum_{i=1}^nu_iu^{kk}u_{ikk}\right)\\
&\quad -\frac{1}{2n}uu^{kk}\left(\frac{1}{2}(u_{11})^{-\frac{1}{2}}\tau_ku_{11}-\frac{1}{4}(u_{11})^{-\frac{3}{2}}(u_{11k})^2+2u_{kk}^2\right)+u^2+2uu_t.
	\end{align*}
Plugging \eqref{ukt} and \eqref{ukkt} into the above inequality, we are led to
\begin{align*}
	0\le& \frac{1}{2}(u_{11})^{-\frac{1}{2}}\left(-\frac{1}{2n}uu^{ii}u^{jj}(u_{ij1})^2+\frac{1}{n}u_1u^{ii}u_{ii1}+\frac{1}{2n}u_{11}\log(\det D^2u)\right.\\
	&\left.+\frac{n+2}{2n}u_{11}\log u+\frac{n+2}{2n}u_{11}+\frac{n+2}{2n}\frac{u_1^2}{u}\right)+u^2+2uu_t\\
	&\quad-\frac{1}{2n}uu^{kk}\left(\frac{1}{2}(u_{11})^{-\frac{1}{2}}\tau_ku_{11}-\frac{1}{4}(u_{11})^{-\frac{3}{2}}(u_{11k})^2+2u_{kk}^2\right)\\
	&\qquad +|Du|^2\left(\frac{1}{n}\log(\det D^2u)+\frac{n+2}{n}\log u+\frac{n+2}{n}\right).
\end{align*}
	In view of \eqref{phik} and assume $ u_{11}>1 $ at $ (0,\bar{t}) $, we get
	\begin{align*}
		0\le& -\frac{1}{2n}uu^{ii}u^{jj}(u_{ij1})^2+\frac{1}{n}u_1u^{ii}u_{ii1}+\frac{1}{2n}u_{11}\log(\det D^2u)\\
		&\quad +\frac{n+2}{2n}u_{11}\log u+\frac{n+2}{2n}u_{11}+\frac{n+2}{2n}\frac{u_1^2}{u}\\
		&\qquad -\frac{1}{2n}uu_{11}u^{kk}\tau_k+\frac{4}{n}u\sum_{k=1}^nu_k^2u_{kk}-\frac{2}{n}u\sqrt{u_{11}}\sum_{k=1}^nu_{kk}\\
		&\qquad\quad  +2\sqrt{u_{11}}|Du|^2\left(\frac{1}{n}\log(\det D^2u)+\frac{n+2}{n}\log u+\frac{n+2}{n}\right)\\
		&\qquad\qquad
		+2\sqrt{u_{11}}u^2+4\sqrt{u_{11}}u\left(\frac{1}{2n}u\log(\det D^2u)+\frac{n+2}{2n}u\log u\right).
	\end{align*}
	In fact, by
	\begin{eqnarray*}
		\left|\frac{1}{n}u_1u^{ii}u_{ii1}\right|\le \frac{u_1^2}{2u}+\frac{1}{2n}uu^{ii}u^{ii}(u_{ii1})^2,
	\end{eqnarray*}
    we may get
    $\displaystyle -\frac{1}{2n}uu^{ii}u^{jj}(u_{ij1})^2+\frac{1}{n}u_1u^{ii}u_{ii1}\le \frac{u_1^2}{2u}.$
	Furthermore, we observe that $\displaystyle -\frac{1}{2n}uu_{11}u^{kk}\tau_k\le 0 $, $\displaystyle \frac{4}{n}u\sum_{k=1}^n u_k^2u_{kk}\le\frac{4}{n}u|Du|^2u_{11} $ and $\displaystyle -\frac{2}{n}u\sqrt{u_{11}}\sum_{k=1}^n u_{kk}\le -\frac{2}{n}u(u_{11})^\frac{3}{2} $. It follows that
	\begin{align*}
		\frac{2}{n}u\sqrt{u_{11}}\le& \frac{u_1^2}{2u}+\frac{1}{2n}\log(\det D^2u)+\frac{n+2}{2n}\log u\\
		&\quad +\frac{n+2}{2n}+\frac{n+2}{2n}\frac{u_1^2}{u}+\frac{4}{n}u|Du|^2\\
		&\qquad +4|Du|^2\left(\frac{1}{2n}\log(\det D^2u)+\frac{n+2}{2n}\log u+\frac{n+2}{2n}\right)\\
		&\qquad\quad +2u^2+4u\left(\frac{1}{2n}u\log(\det D^2u)+\frac{n+2}{2n}u\log u\right).
	\end{align*}
	Employing Lemmas \ref{C0}-\ref{gauss}, we obtain the following inequality
	\begin{align*}
		\sqrt{u_{11}}\le L(1+\log u_{11}),
	\end{align*}
	which implies
	\begin{eqnarray*}
		u_{11}\le L,
	\end{eqnarray*}
	where $ L $ is a positive constant depending only on  $ n $, $  r_0 $ and  $ R_0 $. Recalling that we already have an upper bound on the Gauss curvature (Lemma \ref{gauss}), thus the proof is completed.
\end{proof}
Due to the above a priori estimates, it is clear that the convexity of the hypersurface $ \mathcal{M}_t $ is preserved under the flow \eqref{euclidean} and the solution $ X(\cdot,t) $ maintains uniformly convex.

According to Lemmas \ref{C0}-\ref{2destimate}, the flow \eqref{sevo} is uniformly parabolic and $\displaystyle \left|\frac{\partial s}{\partial t}\right|_{L^\infty(\mathbb{S}^n\times[0,T])}\le L $. It follows from the results of Krylov and Safonov \cite{safonov-krylov}  that the H\"{o}lder continuity estimates for $ \bar{\nabla}^2s $ and $ \frac{\partial s}{\partial t} $ can be obtained. Estimates for higher order derivatives then follow from the bootstrap argument and using the Schauder estimates (see \cite{krylov}). In view of Lemma \ref{C0}, $ R_0 $ and $ r_0 $ are always bounded in finite time interval.
Hence the maximal time $T=\infty$. The uniqueness of the smooth solution $ s(\cdot, t) $ follows by the parabolic comparison principle.
Thus we finish the proof of Theorem \ref{Uin-Exi}.
\section{Some results of the flow in Euclidean geometry}\label{sec-Euclidean Results}
It is of interest to study the flow \eqref{flow} from the perspective of Euclidean geometry. In this section, we study the self-similar solutions of the flow in Euclidean setting. Furthermore, by utilizing the motion characteristics of ellipsoids under this flow, we can describe the rough asymptotic behavior of the flow.
\subsection{The self-similar solutions}
Let us first discuss contracting self-similar solutions of the flow \eqref{euclidean}. We will follow the method used in \cite{Huisken1}. Suppose the initial hypersurface $ \mathcal{M}_0 $ satisfies the equation	$ \displaystyle\frac{s_0}{2n}\log\left(\frac{K_0}{s_0^{n+2}}\right)=\left<X_0, x\right> $, where $ s_0 $ and $ K_0 $ are the support function and Gauss curvature of the hypersurface $ \mathcal{M}_0 $, respectively. Then the homothetic deformation is given by
\begin{align*}
	X(x,t)=h(t)X(x,0),
\end{align*}
where $h(t)$ is a positive function with $h(0)=1$.
Next let us look for the function $ h(t) $. On the one hand, the hypersurface $ X(x,t) $ satisfies the evolution equation
\begin{align*}
	\left<\frac{d}{dt}X(x,t),x\right>=h'(t)\frac{s_0}{2n}\log\left(\frac{K_0}{s_0^{n+2}}\right).
\end{align*}
On the other hand, since it is a solution of the flow \eqref{euclidean}, we obtain the equation
\begin{align}\label{hd1}
	h'(t)\frac{s_0}{2n}\log\left(\frac{K_0}{s_0^{n+2}}\right)=\frac{s}{2n}\log\left(\frac{s^{n+2}}{K}\right).
\end{align}
Notice that we have $ s=hs_0 $ and $ K=h^{-n}K_0 $, by substituting them in \eqref{hd1} and simplifying, we obtain the equation
\begin{align*}
	(2n+2)h\log h=(h+h')\log\left(\frac{K_0}{s_0^{n+2}}\right).
\end{align*}
If $ h(t)+h'(t)\neq0 $, we have
\begin{align*}
	\frac{(2n+2)h\log h}{h+h'}=\log\left(\frac{K_0}{s_0^{n+2}}\right),
\end{align*}
which requires  that $ K_0/s_0^{n+2} $ is a positive constant, so the hypersurface $ X(x,0) $ is an ellipsoid centered at the origin. Otherwise, solving the equation with the initial value $ h(0)=1 $ yields $ h(t)=\exp(-t) $. Thus,  the homothetic deformation given by
\begin{align*}
	X(x,t)=\exp(-t)X(x,0)
\end{align*}
satisfies \eqref{euclidean} up to tangential diffeomorphisms. Our self-similarity condition therefore corresponds to the elliptic equation
\begin{align*}
	\frac{s}{2n}\log\left(\frac{K}{s^{n+2}}\right)=\left<X,x\right>.
\end{align*}
Noting that $ s=\left<X,x\right>$, it is equivalent to
\begin{align*}
	\exp\left(-\frac{2n}{n+2}\right)K^{\frac{1}{n+2}}=\left<X,x\right>.
\end{align*}
Thus, the hypersurface $ X(x,t) $ is a family of ellipsoids centered at the origin (see \cite{Calabi}).
The expanding self-similar solutions of \eqref{euclidean} can be discussed in a similar manner. Therefore, we conclude that
\begin{thm}
	The only strictly convex closed smooth self-similar (both shrinking and expanding) solutions of the flow \eqref{euclidean} are ellipsoid centered at the origin.
\end{thm}
\subsection{The asymptotic behavior of the flow in Euclidean setting}
We begin by providing an interesting example.
\begin{exmp}\label{egellipsoid}
	Suppose that the initial hypersurface $ X_0 $ is an ellipsoid with positive constant equi-affine support function $\rho_0$. Then the solution to \eqref{mainflow1} is a family of ellipsoids given by
	\begin{align*}
		X(\cdot,t)=\rho_0^{\frac{n+2}{2(n+1)}\left(\exp\left(\frac{n+1}{n}t\right)-1\right)}X_0.
	\end{align*}
	Note that $X(\cdot,t)$ converges to the ellipsoid with volume equal to that of $\mathbb{S}^n$ as $t\rightarrow-\infty$.
\end{exmp}
We observe that if the initial ellipsoid $X_0$ has an equi-affine support function $\rho_0$ greater than $1$, then the solution $X(\cdot,t) $ remains to be ellipsoids and the flow \eqref{mainflow1} expands to inﬁnity as $t\rightarrow\infty$. If its equi-affine support function $\rho_0$ is less than $1$, then the flow \eqref{mainflow1} shrinks to a point as $t\rightarrow\infty$.
Therefore, employing the comparison principle, we can conclude that
\begin{prop}\label{roughasy}
	If the initial hypersurface $ X_0 $ is sufficiently small, then the solution $ X(\cdot,t) $ of \eqref{euclidean} will shrink to a point when $ t $ goes to infinity. If the initial hypersurface $ X_0 $ is appropriately large,  then the solution $ X(\cdot,t) $ of \eqref{euclidean} will expand to infinity as $ t $ goes to infinity.
\end{prop}
\section{Estimates on the centro-affine invariants}\label{sec-UniformE}
In this section, we derive the uniform estimates on centro-affine invariants, which will be used in the next section.
\subsection{The estimates for Tchebychev vector field}
We first provide the uniform estimate for Tchebychev vector field $ T $. Utilizing Proposition \ref{evo}, a direct computation shows
	\begin{align}\label{Tevo}
	\frac{\partial}{\partial t}|T|^2=T^iH_i+2\left(1+\frac{1}{n}\right)|T|^2-C^{ijk}T_iT_jT_k.
\end{align}
To proceed, we want to write this evolution equation in a parabolic form. To achieve this, we compute
\begin{align*}
	\Delta|T|^2&=2T^kg^{ij}T_{k;ij}+2|\nabla T|^2\\
	&=2T^kg^{ij}(T_{i;jk}-R^\alpha_{jki}T_\alpha)+2|\nabla T|^2\\
	&=2T^kg^{ij}[T_{i;jk}-(C_{ij}^pC_{pk}^\alpha-C_{ik}^pC_{jp}^\alpha+g_{ik}\delta_j^\alpha-g_{ij}\delta_k^\alpha)T_\alpha]+2|\nabla T|^2\\
	&=2nT^kH_k-2nT^kT^pT_\alpha C_{pk}^\alpha+2T^kT_\alpha C_k^{jp}C_{jp}^\alpha+2(n-1)|T|^2+2|\nabla T|^2.
\end{align*}
Hence,
	\begin{align}\label{T2pf}
	\frac{\partial}{\partial t}|T|^2=\frac{1}{2n}\Delta|T|^2-\frac{1}{n}|\nabla T|^2+\frac{n+3}{n}|T|^2-\frac{1}{n}T_iT_\alpha C^{ipl}C^{\alpha}_{pl}.
\end{align}
\begin{lem}\label{Tbound}
	Under the flow \eqref{flow}, the quantity  $ |T(\cdot,t)|^2 $ of the hypersurface $\mathcal{M}_t$ is uniformly bounded on $ [0,+\infty) $. More precisely, we have the following inequality
	\begin{eqnarray*}
		|T(\cdot,t)|^2\le\max\left\{\frac{n+3}{n}, \max_{\mathcal{M}_0}|T|^2\right\}, \qquad \forall t \in [0,+\infty).
	\end{eqnarray*}
\end{lem}
\begin{proof}
	Let $ Q_{ij}:=T_pC^p_{ij} $. Then \eqref{T2pf} becomes
	\begin{align*}
		\frac{\partial}{\partial t}|T|^2=\frac{1}{2n}\Delta|T|^2-\frac{1}{n}|\nabla T|^2+\frac{n+3}{n}|T|^2-\frac{1}{n}|Q|^2.
	\end{align*}
Since $ n|T|^2=Q_i^i $ and thus Cauchy-Schwarz inequality applied to the eigenvalues of $ Q $ implies $ |Q|^2\ge n|T|^4 $. Then we obtain
	\begin{align}\label{Tinequality}
		\frac{\partial}{\partial t}|T|^2\le\frac{1}{2n}\Delta|T|^2+\frac{n+3}{n}|T|^2-|T|^4.
	\end{align}
	Let $ Z(t):=\max_{\mathcal{M}_t}|T|^2 $. If $ Z(t)\le Z(0) $, the conclusion clearly holds.
	Otherwise, there exists some $ t>0 $ such that $ \beta:=Z(t)>Z(0) $. Suppose that $ \beta $ is achieved for the first time at $ t_0 $. Then we can see that  $ \frac{d}{dt}Z\ge 0 $ at $ t_0 $. Hence, it follows from \eqref{Tinequality} that $ \beta $ satisfies
	\begin{align*}
		\beta\left(\frac{n+3}{n}-\beta\right)\ge 0,
	\end{align*}
	which leads to $ \beta\le\frac{n+3}{n} $. Thus we complete the proof.
\end{proof}
\subsection{Estimates for the cubic form}
Let us now deduce an uniform estimate for cubic form $ C $. We begin by computing the Laplacian of the cubic form $ C $.
\begin{align*}
	\Delta C_{ijk}=&g^{lm}C_{ijk;lm}=g^{lm}C_{ijl;km}\\=&g^{lm}[C_{ijl;mk}-(R_{mki}^\alpha C_{\alpha jl}+R^\alpha_{mkj}C_{i\alpha l}+R^\alpha_{mkl}C_{ij\alpha})]\\
	=&nT_{i;jk}-g^{lm}(R^\alpha_{mki}C_{\alpha jl}+R_{mkj}^\alpha C_{i\alpha l}+R^\alpha_{mkl}C_{ij\alpha})\\
=&nT_{i;jk}-g^{lm}C_{\alpha jl}(C_{im}^pC_{pk}^\alpha-C_{ik}^pC_{pm}^\alpha+g_{ik}\delta_m^\alpha-g_{im}\delta_k^\alpha)\\
&\quad -g^{lm}C_{i\alpha l}(C_{jm}^pC_{pk}^\alpha-C_{jk}^pC_{pm}^\alpha+g_{jk}\delta_m^\alpha-g_{jm}\delta_k^\alpha)\\
&\qquad -g^{lm}C_{ij\alpha}(C_{lm}^pC_{pk}^\alpha-C_{lk}^pC_{pm}^\alpha+g_{lk}\delta_m^\alpha-g_{lm}\delta_k^\alpha)\\
=&nT_{i;jk}-C_{\alpha j}^mC_{im}^pC_{pk}^\alpha+C_{\alpha j}^mC_{ik}^pC_{pm}^\alpha-nT_jg_{ik}+C_{kj}^mg_{im}
	\\&\quad-C_{i\alpha}^mC_{jm}^pC_{pk}^\alpha+C_{i\alpha}^mC_{jk}^pC_{pm}^\alpha-nT_ig_{jk}+C_{ik}^mg_{jm}
	\\&\qquad -g^{lm}C_{ij\alpha}C_{lm}^pC_{pk}^\alpha+g^{lm}C_{ij\alpha}C_{lk}^pC_{pm}^\alpha-\delta_k^m\delta_m^\alpha C_{ij\alpha}+nC_{ijk}\\
  =&nT_{i;jk}-2C_{\alpha j}^mC_{im}^pC_{pk}^\alpha\\
  &\quad +C_{\alpha j}^mC_{ik}^pC_{pm}^\alpha
	+C_{i\alpha}^mC_{jk}^pC_{pm}^\alpha+C_{ij}^\alpha C_{kp}^mC_{m\alpha}^p
	\\&\qquad -nT_ig_{jk}-nT_jg_{ik}-nT_pC_{ij}^\alpha C_{k\alpha}^p+(n+1)C_{ijk}.
\end{align*}
Together with the evolution equation of $ C $ (Proposition \ref{evo} (5)), we obtain
\begin{align}
	\nonumber\frac{\partial}{\partial t}C_{ijk}=&\frac{1}{2n}\Delta C_{ijk}+\frac{1}{n}C_{\alpha j}^mC_{im}^pC_{pk}^\alpha-\frac{n+1}{2n}C_{ijk}\\
	&\nonumber\quad -\frac{1}{2n}(C_{\alpha j}^mC_{ik}^pC_{pm}^\alpha+C_{i\alpha}^mC_{jk}^pC_{pm}^\alpha+C_{ij}^\alpha C_{kp}^mC_{m\alpha}^p)\\
	&\nonumber\qquad +\frac{1}{2}(T_\alpha C_{ik}^pC_{pj}^\alpha+T_\alpha C_{jk}^pC_{ip}^\alpha+T_pC_{ij}^\alpha C_{k\alpha}^p)\\
	&\label{n=0e2}\qquad \quad +T_ig_{jk}+T_jg_{ik}+T_kg_{ij}.
\end{align}
Now we proceed to represent $ |C|^2 $ in a parabolic form. To do this, we  calculate:
\begin{align*}
	\frac{\partial}{\partial t}|C|^2=&\frac{\partial}{\partial t}(C_{ijk}g^{iq}g^{jr}g^{ks}C_{qrs})=3C_{ijk}\left(\frac{\partial}{\partial t}g^{iq}\right)C_q^{jk}+2\left(\frac{\partial}{\partial t}C_{ijk}\right)C^{ijk}
	\\
	=&-3T_pC^{piq}C_{ijk}C_q^{jk}+2C^{ijk}\left[\frac{1}{2n}\Delta C_{ijk}+\frac{1}{n}C_{\alpha j}^mC_{im}^pC_{pk}^\alpha\right.\\
	&\quad -\frac{n+1}{2n}C_{ijk}-\frac{1}{2n}(C_{\alpha j}^mC_{ik}^pC_{pm}^\alpha+C_{i\alpha}^mC_{jk}^pC_{pm}^\alpha+C_{ij}^\alpha C_{kp}^mC_{m\alpha}^p)\\
	&\qquad \left.+\frac{1}{2}(T_\alpha C_{ik}^pC_{pj}^\alpha+T_\alpha C_{jk}^pC_{ip}^\alpha+T_pC_{ij}^\alpha C_{k\alpha}^p)+T_ig_{jk}+T_jg_{ik}+T_kg_{ij}\right]\\
	=&\frac{1}{n}C^{ijk}\Delta C_{ijk}+\frac{2}{n}C^{ijk}C_{\alpha j}^mC_{im}^pC_{pk}^\alpha-\frac{n+1}{2n}|C|^2+6n|T|^2\\
&\qquad -\frac{1}{n}C^{ijk}(C_{\alpha j}^mC_{ik}^pC_{pm}^\alpha+C_{i\alpha}^mC_{jk}^pC_{pm}^\alpha+C_{ij}^\alpha C_{kp}^mC_{m\alpha}^p).
\end{align*}
Let $ P_{ij}:=C_{il}^kC_{jk}^l $, and note that $\Delta|C|^2=2\Delta C_{ijk}C^{ijk}+2|\nabla C|^2$, we arrive at
\begin{align*}
	\frac{\partial}{\partial t}|C|^2=&\frac{1}{2n}\Delta|C|^2-\frac{1}{n}|\nabla C|^2+\frac{2}{n}C^{ijk}C_{\alpha j}^mC_{im}^pC_{pk}^\alpha
	\\&-\frac{n+1}{n}|C|^2+6n|T|^2-\frac{3}{n}|P|^2.
\end{align*}
Let $ Q_{ijkl}:=C_{ij}^mC_{klm}-C_{ik}^mC_{jlm} $, we find
\begin{align*}
	0\le\frac{1}{2}|Q|^2=|P|^2-C_{il}^mC_{km}^pC_{pj}^lC^{ijk},
\end{align*}
and so
\begin{align*}
	\frac{\partial}{\partial t}|C|^2&\le\frac{1}{2n}\Delta|C|^2+\frac{2}{n}C^{ijk}C_{\alpha j}^mC_{im}^pC_{pk}^\alpha-\frac{3}{n}|P|^2+6n|T|^2\\
	&\le\frac{1}{2n}\Delta|C|^2-\frac{1}{n}|P|^2+6n|T|^2
	\\&\le\frac{1}{2n}\Delta|C|^2-\frac{1}{n^2}|C|^4+6n|T|^2,
\end{align*}
since $ |C|^2=P^i_i $, and thus Cauchy-Schwartz inequality applied to the eigenvalues of $ P $ implies $ |P|^2\ge\frac{1}{n}|C|^4 $. By the maximum principle, this implies the following inequality for $ \rho(t):=\max_{\mathcal{M}_t}|C|^2 $
\begin{align*}
	\frac{d\rho}{d t}\le-\frac{1}{n^2}\rho^2+6n|T|^2.
\end{align*}
From Lemma \ref{Tbound}, we know that $ |T|^2 $ is uniformly bounded on $ [0,+\infty) $. Hence the behavior of $ \rho $ when large is modeled by the differential inequality
\begin{align}\label{din}
	\frac{d\rho}{d t}\le-\frac{1}{n^2}\rho^2.
\end{align}
We thus deduce that
\begin{align*}
	\rho(t)\le L+\left(\frac{1}{\rho(0)}+\frac{1}{n^2}t\right)^{-1}, \quad \forall t\in (0,+\infty)
\end{align*}
for some positive constants $ L$.
From this, we can conclude that
\begin{lem}\label{Cbound}
	Under the flow \eqref{flow}, the quantity  $ |C(\cdot,t)|^2 $ of the hypersurface $ \mathcal{M}_t $ is uniformly bounded on $ [0,+\infty) $.
\end{lem}
\subsection{Higher derivatives of the cubic form}\label{hdc}
In this subsection, we will derive the higher order derivative estimates for the cubic form $ C $ using maximum principle arguments. When dealing with two tensors $A$ and $B$, we denote $ A*B $ as a linear combination of tensors formed by contraction on $ A $ and $ B $ by centro-affine metric $ g $. Furthermore, we will represent the $m$th iterated covariant derivative of a tensor $A$ as $ \nabla^mA $.

We aim to derive the evolution equation for the $ m $th covariant derivative $ \nabla^mC $ of the cubic form $ C $. To this end, we find the following lemma to be helpful.
\begin{lem}
	If $ A $ and $ B $ are tensors satisfying the evolution equation
	\begin{align*}
		\frac{\partial}{\partial t}A=\Delta A+B,
	\end{align*}
then the covariant derivative $ \nabla A $ satisfies an equation of the form
\begin{align*}
	\frac{\partial}{\partial t}\nabla A=\Delta\nabla A+C*\nabla C*A+C*C*\nabla A+\nabla B.
\end{align*}
\end{lem}
\begin{proof}
	The covariant derivative $ \nabla $ involves the Christoffel symbols $ \Gamma_{ij}^k $. In view of Proposition \ref{evo}, their time derivatives are
	\begin{align*}
	\frac{\partial}{\partial t}\Gamma_{ij}^k=\frac{1}{2}g^{kl}\left\{\nabla_j\left(\frac{\partial}{\partial t}g_{il}\right)+\nabla_i\left(\frac{\partial}{\partial t}g_{jl}\right)-\nabla_l\left(\frac{\partial}{\partial t}g_{ij}\right)\right\}=C*\nabla C.
	\end{align*}
Thus
\begin{align*}
	\frac{\partial}{\partial t}\nabla A=\nabla\left(\frac{\partial}{\partial t}A\right)+C*\nabla C*A.
\end{align*}
Then interchanging derivatives yields
\begin{align*}
	\nabla\Delta A=\Delta\nabla A+C*\nabla C*A+C*C*\nabla A,
\end{align*}
and this completes the proof.
\end{proof}
\begin{prop}
	The $ m $th covariant derivative $ \nabla^mC $ of the cubic form $ C $ satisfies an evolution equation of the form
	\begin{align*}
		\frac{\partial}{\partial t}\nabla^mC=\frac{1}{2n}\Delta\nabla^mC+\nabla^mC+\sum_{i+j+k=m}^{}\nabla^iC*\nabla^jC*\nabla^kC.
	\end{align*}
\end{prop}
\begin{proof}
	In the case of $ m=0 $ it holds true by \eqref{n=0e2}, that is,
	\begin{align*}
		\frac{\partial}{\partial t}C=\frac{1}{2n}\Delta C+C+C*C*C.
	\end{align*}
We proceed by induction on $ m $. Employing the previous lemma, we immediately have
\begin{align*}
	\frac{\partial}{\partial t}\nabla^mC=&\Delta\nabla^mC+C*\nabla C*\nabla^{m-1}C+C*C*\nabla^mC\\
	&\qquad +\nabla^mC+\sum_{i+j+k=m}^{}\nabla^iC*\nabla^jC*\nabla^kC\\
	=&\Delta\nabla^mC+\nabla^mC+\sum_{i+j+k=m}^{}\nabla^iC*\nabla^jC*\nabla^kC,
\end{align*}
This completes the induction.
\end{proof}
\begin{cor}\label{mevo}
	For any $ m\ge0 $ we have the evolution equation
	\begin{align*}
		\frac{\partial}{\partial t}|\nabla^mC|^2=&\frac{1}{2n}\Delta|\nabla^mC|^2-\frac{1}{n}|\nabla^{m+1}C|^2\\
		&+\nabla^mC*\nabla^mC+\sum_{i+j+k=m}^{}\nabla^iC*\nabla^jC*\nabla^kC*\nabla^mC.
	\end{align*}
\end{cor}
\begin{proof}
This follows from the previous proposition. Indeed, We have
\begin{align*}
	\frac{\partial}{\partial t}|\nabla^mC|^2=2\left<\nabla^mC,\frac{\partial}{\partial t}\nabla^mC\right>+C*C*\nabla^mC*\nabla^mC,
\end{align*}
where the extra terms come from the variation of $ g^{ij} $ defining the norm $ |\cdot|$. On the other hand, a direct computation gives
\begin{align*}
	\Delta|\nabla^mC|^2=2\left<\nabla^mC,\Delta\nabla^mC\right>+2|\nabla^{m+1}C|^2,
\end{align*}
and the result follows.
\end{proof}
\begin{prop}\label{Cmbound}
	For each $ m\ge0 $ there is a positive constant $ L_m $ such that
	\begin{align*}
		|\nabla^mC|^2\le L_m
	\end{align*}
uniformly on $ [0,+\infty) $.
\end{prop}
\begin{proof}
	In view of Lemma \ref{Cbound} we know that $ |C|^2=|\nabla^0C|^2 $ are uniformly bounded and we prove Proposition \ref{Cmbound} by induction on $ m $. Suppose $ |\nabla^{\bar{m}}C|^2 $ is uniformly bounded by $ L_m $ for all $ \bar{m}\le m $. Let $ t_0\in[0,+\infty) $ and we put $ \tilde{t}=t-t_0 $ for $ t\in[t_0,t_0+1] $ so that $ \tilde{t}\in[0,1] $. Now let us consider the function
	\begin{align*}
		f:=\tilde{t}|\nabla^{m+1}C|^2+N|\nabla^mC|^2,
	\end{align*}
where $ N $ is a positive constant we shall choose in a minute. The evolution equation for $ f $ follows from Corollary \ref{mevo}
\begin{align*}
	\frac{\partial}{\partial t}f=&|\nabla^{m+1}C|^2+\frac{1}{2n}\Delta f+\tilde{t}\left(-\frac{1}{n}|\nabla^{m+2}C|^2+\nabla^{m+1}C*\nabla^{m+1}C\right.\\
	&\;\;\left.+\sum_{i+j+k=m+1}^{}\nabla^iC*\nabla^jC*\nabla^kC*\nabla^{m+1}C\right)+N\left(-\frac{1}{n}|\nabla^{m+1}C|^2\right.\\
	&\qquad \left.+\nabla^{m}C*\nabla^{m}C+\sum_{i+j+k=m}^{}\nabla^iC*\nabla^jC*\nabla^kC*\nabla^{m}C\right).
\end{align*}
Since $ \tilde{t}\le1 $, we can deduce from the induction hypothesis that
\begin{align*}
		\frac{\partial}{\partial t}f\le\frac{1}{2n}\Delta f+L_{11}(1+|\nabla^{m+1}C|^2)-\frac{N}{n}|\nabla^{m+1}C|^2+NL_{22},
\end{align*}
where $ L_{11} $ and $ L_{22} $ are constants depending on $ L_m $. Now we choose an appropriate constant $ N $ such that
\begin{align*}
L_{11}	-\frac{N}{n}\le 0.
\end{align*}
Then
\begin{align*}
	\frac{\partial}{\partial t}f\le\frac{1}{2n}\Delta f+L_{33},
\end{align*}
where $ L_{33} $ is a constant depending on $ L_m $. Note that $ f\le NL_m $ at $ t=t_0 $. Hence by the maximum principle
\begin{align*}
	f\le L_{33}+NL_m
\end{align*}
on $[t_0,t_0+1] $.
It follows that
\begin{align*}
	|\nabla^{m+1}C|^2\le L_{33}+NL_m
\end{align*}
at $ t=t_0+1 $.
Recalling $ t_0\in[0,+\infty) $ was arbitrary. Thus we complete the proof.
\end{proof}
Because any space-time derivative $ \frac{\partial^j}{\partial t^j}\nabla^kC $ can express solely in terms of space derivatives of $C$, we can conclude that
\begin{cor}\label{space-time}
	There exist constants $ L_{j,k} $ such that if $|C|$ is bounded then the space-time derivatives are bounded, that means
	\begin{align*}
		 \left|\frac{\partial^j}{\partial t^j}\nabla^kC\right|^2\le L_{j,k}.
	\end{align*}
\end{cor}
\section{Proof of Theorem \ref{mainthmA}}\label{sec-ProofA}
In this section, our goal is to prove the convergence result of the flow \eqref{flow}. To do so, we begin by introducing the notion of the centro-affine area, which plays a crucial role in the proof. For  a closed uniformly convex hypersurface $ \mathcal{M} $, the centro-affine area is defined as follows
\begin{eqnarray*}
	\text{Area}(\mathcal{M}):=\int_{\mathcal{M}}\sqrt{\det(g_{ij})}d\mu_{\mathcal{M}},
\end{eqnarray*}
where $ \sqrt{\det(g_{ij})}d\mu_M $ denotes the area form with respect to the cntro-affine metric $ g_{ij} $ on $ \mathcal{M} $.
Similar to the isoperimetric inequality in the Euclidean geometry, there exists the following well-known inequality concerning the centro-affine area, which plays a key role in our subsequent arguments.
\begin{prop}[The centro-affine isoperimetric inequality \cite{LiAM}]
	In centro-affine geometry, for any closed smooth strictly convex hypersurface $ \mathcal{M} $ with the origin in its interior, the centro-affine area of the hypersurface satisfies
	\begin{eqnarray}\label{caiso}
		\text{Area}(\mathcal{M})\le O_n,
	\end{eqnarray}
		where $ O_n $ is the n-dimensional volume of the Euclidean unit ball in $ \R^{n+1} $,	and the equality holds if and only if $ \mathcal{M} $ is an ellipsoid centered at the origin.
\end{prop}
Let $d\mu_t:=\sqrt{\det(g_{ij})}d\mu_{\mathcal{M}_t}$. From Proposition \ref{evo}, we know that the evolution equation of the centro-affine area under the flow \eqref{flow} is
\begin{eqnarray}\label{Aevo}
	\frac{d}{dt}\text{Area}(\mathcal{M}_t)=\int_{\mathcal{M}_t}\frac{\partial}{\partial t}\sqrt{\det(g_{ij})}d\mu_{\mathcal{M}_t}=\frac{n}{2}\int_{\mathcal{M}_t}|T|^2d\mu_t,
\end{eqnarray}
which implies that $ \text{Area}(\mathcal{M}_t) $ is non-decreasing along this flow.
Furthermore, referring to \eqref{rel-support}, we can deduce the following conclusion.
\begin{prop}\label{monton}
	Let $ \mathcal{M}_t $ be a closed smooth strictly convex solution of the flow \eqref{flow}. Then the centro-affine area $ Area(\mathcal{M}_t) $ is non-decreasing along the flow. The monotonicity is strict unless $ \mathcal{M}_t $ is an ellipsoid centered at the origin.
\end{prop}
By the centro-affine isoperimetric inequality \eqref{caiso} and the evolution equation \eqref{Aevo}, it is evident that
\begin{lem}\label{intc}
	Under the flow \eqref{flow}, $ \int_{0}^{+\infty}(\int_{\mathcal{M}_t}|T|^2d\mu_t)dt $ converges.
\end{lem}
Based on the results from Section \ref{sec-UniformE}, we obtain
\begin{prop}\label{to0}
	Under the flow \eqref{flow}, $ \int_{\mathcal{M}_t}|T|^2d\mu_t $ converges to zero as time tends to infinity, that is,
	\begin{align*}
		\lim_{t\rightarrow+\infty}\int_{\mathcal{M}_t}|T|^2d\mu_t=0.
	\end{align*}
\end{prop}
\begin{proof}
	Since
	\begin{align*}
		\frac{d}{dt}\int_{\mathcal{M}_t}|T|^2d\mu_t=\int_{\mathcal{M}_t}\left(\frac{\partial}{\partial t}|T|^2+\frac{n}{2}|T|^4\right)d\mu_t.
	\end{align*}
In view of \eqref{Tevo}, \eqref{caiso} and Proposition \ref{Cmbound}, $ \displaystyle\frac{d}{dt}\int_{\mathcal{M}_t}|T|^2d\mu_t $ is bounded on $ [0,+\infty) $.
Thus the result immediately follows from Lemma \ref{intc}.
\end{proof}
Now let us proceed to complete the proof of Theorem \ref{mainthmA}.
\begin{proof}[Proof of Theorem \ref{mainthmA}]
	In view of Corollary \ref{space-time}, all space-time derivative of $ |T|^2 $ are uniformly bounded. Thus, according to the Arzel\`{a}-Ascoli theorem, for any sequence of times $ \{t_k\} $ there exists a subsequence $ \{t_{k_i}\} $ such that $ |T(x,t_{k_i})|^2 $ converges smoothly to a function $ h(x) $ as $ t_{k_i} \rightarrow +\infty $. In view of Proposition \ref{to0}, $ h(x) $ must be zero, i.e., $ h(x)\equiv0 $. Since every subsequence converges to zero, $ |T|^2 $ converges smoothly to zero as $ t\rightarrow +\infty $. Hence we complete the proof.
\end{proof}
\section{The eternal solutions of the flow}\label{sec-EternalSolutions}
In this section, we discuss the eternal solutions of the flow \eqref{flow}. Assume from now on that $ X(\cdot,t) $ is a solution to the flow \eqref{flow} defined on $ (-\infty,+\infty) $ with $ |T|^2 $ is uniformly bounded on $ (-\infty,+\infty) $. Under this assumption, we  are able to obtain a conclusion similar to that of Lemma \ref{Cbound}.
\begin{lem}\label{esCbound}
	Under the flow \eqref{flow}, the quantity  $ |C(\cdot,t)|^2 $ of the hypersurface $ \mathcal{M}_t $ is uniformly bounded on $ (-\infty,+\infty) $.
\end{lem}
\begin{proof}
	The proof is similar. For any $ t\in (t_0,+\infty)$, the differential inequality \eqref{din} gives
	\begin{align*}
		\rho(t)\le L+\frac{n^2}{t-t_0}
	\end{align*}
for some positive constants $ L $.
Since $ X(\cdot,t) $ is an eternal solution we may let $ t_0\rightarrow-\infty $, and the result follows.
\end{proof}
Then, the proof of the following proposition is a simple modification of that of Proposition \ref{Cmbound} and is omitted.
\begin{prop}\label{Cm-bound}
	For each $ m\ge 0 $ there is a positive constant $ L_m $ such that
	\begin{align*}
		|\nabla^mC|^2\le L_m
	\end{align*}
uniformly on $ (-\infty,+\infty) $.
\end{prop}
\begin{cor}\label{eternalspace-time}
	There exist constants $ L_{j,k} $ such that
	\begin{align*}
		\left|\frac{\partial^j}{\partial t^j}\nabla^kC\right|^2\le L_{j,k}
	\end{align*}
uniformly on $ (-\infty,+\infty) $.
\end{cor}
Furthermore, we have the following results. The proof follows a similar approach as in the previous section, and therefore, we omit the proof here.
\begin{lem}
	Under the flow \eqref{flow}, $ \int_{-\infty}^{+\infty}(\int_{\mathcal{M}_t}|T|^2d\mu_t)dt $ converges.
\end{lem}
\begin{prop}\label{Tto0-}
	Under the flow \eqref{flow}, $ \int_{\mathcal{M}_t}|T|^2d\mu_t $ converges to zero as time tends to $ -\infty $, that is,
	\begin{align*}
		\lim_{t\rightarrow-\infty}\int_{\mathcal{M}_t}|T|^2d\mu_t=0.
	\end{align*}
\end{prop}
As a consequence of Corollary \ref{eternalspace-time} and Proposition \ref{Tto0-}, the convergence of the Tchebychev vector field $ T $ as $ t\rightarrow -\infty $ can be obtained with a similar argument in the proof of Theorem \ref{mainthmA}.
\begin{prop}\label{thmes}
	Assume that $ X(\cdot,t) $ is a solution to the flow \eqref{flow} defined on $ (-\infty,+\infty) $ with $ |T|^2 $ is uniformly bounded on $ (-\infty,+\infty) $. Then the Tchebychev vector field  $ T $  converges smoothly to zero as $ t\rightarrow -\infty. $
\end{prop}

\begin{proof}[Proof of Theorem \ref{mainthmB}]
	By Theorem \ref{thmes}, $\mathcal{M}_{-\infty}$ is an origin-centered ellipsoids. This in turn implies that $ \lim_{t\rightarrow-\infty}\text{Area}(\mathcal{M}_t)=O_n $. On the other hand, by the centro-affine isoperimetric inequality \eqref{caiso} and Proposition \ref{monton}, for any $ t_1\in(-\infty,+\infty) $, we have
	\begin{align*}
		O_n\ge \text{Area}(\mathcal{M}_{t_1})\ge  \lim_{t\rightarrow-\infty}\text{Area}(\mathcal{M}_t)=O_n.
	\end{align*}
Thus $ \text{Area}(\mathcal{M}_t)\equiv O_n $ on $ (-\infty,+\infty) $. Hence, $ X(\cdot,t) $ is an origin-centered ellipsoid for every time $ t\in (-\infty,+\infty) $.
\end{proof}

\subsection*{Acknowledgements}
This work was supported by by National Natural Science Foundation of China (Grant Nos. 11631007 and 11971251).

\end{document}